\documentclass[11pt]{amsart}
\usepackage[latin1]{inputenc}
\usepackage[english]{babel}
\usepackage[T1]{fontenc}
\usepackage{amsmath,amssymb}
\usepackage{stmaryrd}
\usepackage{enumerate}
\usepackage{amsthm}
\usepackage{mathrsfs}
\usepackage{geometry}
\usepackage{mathtools}
\usepackage{graphicx}
\usepackage[draft,danish]{fixme}
\usepackage[all,line]{xy}
\mathtoolsset{showonlyrefs}
\newtheorem{thm}{Theorem}[section]
\newtheorem{prop}[thm]{Proposition}
\newtheorem{lemma}[thm]{Lemma}
\theoremstyle{definition}
\newtheorem{defin}[thm]{Definition}
\theoremstyle{remark}
\newtheorem{rem}[thm]{Remark}

\newtheorem{claim}[thm]{Claim}

\newcommand{\iso}{\simeq}
\newcommand{\isomap}{\stackrel{\sim}{\to}}
\newcommand{\desc}{{\mathfrak{Desc}}}

\newcommand{\ccal}{\mathcal{C}}
\newcommand{\R}{\mathbb{R}}

\newcommand{\She}{\mathcal{O}}
\newcommand{\Hom}{\text{Hom}}

\newcommand{\rep}{\mathsf{Rep}}
\newcommand{\act}{\; \rotatebox[origin=c]{270}{$\circlearrowright$} \;}

\begin{document}

\title{Demazure descent and representations of reductive groups}
\author{Sergey Arkhipov}
\author{Tina Kanstrup}
\address{S.A. Matematisk Institut, Aarhus Universitet, Ny Munkegade, DK-8000 , Århus C, Denmark, email: hippie@qgm.au.dk} 
\address{T.K. Centre for Quantum Geometry of Moduli Spaces, Aarhus Universitet, Ny Munkegade, DK-8000 , Århus C, Denmark, email: tina@qgm.au.dk}
\maketitle

\begin{abstract}We introduce the notion of Demazure descent data on a triangulated category $\ccal$ and define the descent category for such data. We illustrate the definition by our basic example. Let $G$ be a reductive algebraic group with a Borel subgroup $B$. Demazure functors form Demazure descent data on $D^b (\rep(B))$ and the descent category is equivalent to $D^b (\rep(G))$.
\end{abstract}

\section{Motivation}
The present paper is the first one in series devoted to various cases of categorical descent. Philosophically, our interest in the subject grew out of attempts to understand the main construction from the recent paper by Ben-Zvi and Nadler \cite{BN} in plain terms that would not involve higher category theory. 

\subsection{Beilinson-Bernstein localization and derived descent}
Let $G$ be a reductive algebraic group with the Lie algebra $\mathfrak{g}$. Denote the Flag variety of $G$ by $\mathfrak{Fl}$. A major part of Geometric Representation Theory originated in the seminal work of Beilinson and Bernstein \cite{BB} devoted to  investigation of the globalization functor $\mathsf{D}\operatorname{-}\mathsf{mod}(\mathfrak{Fl})\to U(\mathfrak{g})\operatorname{-}\mathsf{mod}$. This functor turns out to be fully faithful and provides geometric and topological tools to investigate a wide class of $U(\mathfrak{g})$-modules, in particular the ones from the famous category $\mathcal{O}$. Various generalizations of this result lead to investigation of the categories of twisted D-modules on the Flag variety and on the base affine space for $G$, and of their derived categories. 

Ben-Zvi and Nadler define a certain comonad acting on a  higher categorical version for the derived category  of D-modules on the base affine space. In fact, the functor is built into the higher categorical treating of Beilinson-Bernstein localization-globalization construction. 

Using the heavy machinery of Barr-Beck-Lurie descent, the authors  argue that the derived category of $U(\mathfrak{g})$-modules is equivalent to  the category of D-modules equivariant with respect to this comonad. Thus the global sections functor becomes equivariantization with respect to the action. The comonad is called the Hecke comonad. It provides a categorification for the classical action of the Weyl group on various homological and K-theoretic invariants of the Flag variety. 

Notice that the descent construction fails to work on the level of the usual triangulated categories. Ideally one would like to replace it by a categorical action of the Weyl group or rather of the Braid group on categories of D-modules related to the Flag variety. One would  need to define a notion of  "invariants" with respect to such action.

\subsection{Descent in equivariant K-theory}
Another  source of inspiration for the present paper, which is in a way closer to our work,  is a recent article of  Harada,  Landweber and Sjamaar \cite{HLS}. Given a compact space $X$ with an action of a compact reductive Lie group $G$, the authors express the $G$-equivariant K-theory of $X$ via the $T$-equivariant one. Here $T$ denotes a fixed maximal torus in $G$. Harada et al. show that the natural action of the Weyl group $W$ on $K_T(X)$ extends to an action of a degenerate Hecke ring generated by divided difference operators which was introduced earlier in the context of Schubert calculus by Demazure. The operators are called Demazure operators.  

The main result in the paper  \cite{HLS} states that the  ring $K_G(X)$ is isomorphic to the subring of $K_T(X)$ annihilated by the augmentation ideal in the degenerate Hecke algebra. In other words, a $T$-equivariant class is $G$-equivariant if and only if it is killed by the Demazure operators. 

In the present paper, we define a notion of Demazure descent  on a triangulated category $\mathcal{C}$. Thus Demazure operators are replaced by Demazure functors. These functors satisfy a categorified version of degenerate Hecke algebra relations and form a {\it{Demazure descent data}} on $\mathcal{C}$. We define the {\it{descent category}} for such data. Demazure descent is supposed to be a technique replacing the naive notion of Weyl group invariants, on the categorical level.

We provide the first example of Demazure descent. Consider a reductive algebraic group $G$, fix a Borel subgroup $B\subset G$. Categorifying the construction form \cite{HLS}, we consider Demazure functors $D_{s_i}$ acting on the derived category  of $B$-modules.  We prove that the functors form a Demazure descent data and identify the descent category with the derived category of $G$-modules.

\subsection{Acknowledgements}
The authors are grateful to H.H. Andersen, C. Dodd, V. Ginzburg, M. Harada and R. Rouquier for many stimulating discussions. The project started in the summer of 2012 when the first named author visited IHES. S.A. is grateful to IHES for perfect working conditions. Both authors' research was supported in part by center of excellence grants "Centre for Quantum Geometry of Moduli Spaces" and by FNU grant "Algebraic Groups and Applications".

\section{The setting.}
\subsection{Root data.} Let $G$ be a reductive algebraic group over an algebraically closed field $k$ of characteristic zero. Let $T$ be a Cartan subgroup of $G$ and let $(I, X, Y )$ be the corresponding root data, where $I$ is the set of vertices of the Dynkin diagram, $X$ is the weight lattice of $G$ and $Y$ is the coroot lattice of $G$. Choose a Borel subgroup $T\subset B\subset G$. Denote the set of roots for $G$ by $\Phi=\Phi^+\sqcup \Phi^-$. Let $\{\alpha_1, \dots, \alpha_n\}$ be the set of simple roots. The Weyl group $W=\text{Norm}(T)/T$ of the fixed maximal torus acts naturally on the lattices $X$ and $Y$ and on the $\R$-vector spaces spanned by them, by reflections in root hyperplanes. The simple reflection corresponding to an $\alpha_i$ is denoted by $s_i$. The elements $s_1, \ldots, s_n$ form a set of generators for $W$. For $w\in W$ denote the length of a minimal expression of $w$ via the generators by $\ell(w)$. We have a partial ordering on $W$ called the Bruhat ordering. $w' \leq w$ if there exists a reduced expression for $w'$ that can be obtained from a reduced expression for $w$ by deleting a number of simple reflections. 

The monoid $\text{Br}^+$ with generators $\{T_w,\ w\in W\}$ and relations 
$$T_{w_1}T_{w_2}=T_{w_1w_2}  \text{ if } \ell(w_1)+\ell(w_2)=\ell(w_1w_2) \text{ in } W
$$
is called the braid monoid of $G$.

\subsection{Categories of representations.} For an algebraic group $H$, we denote the Hopf algebra of polynomial functions on $H$ by $\mathcal{O}(H)$. Let $\rep(H)$ be the category of $\mathcal{O}(H)$-comodules. This is an  Abelian tensor category.

Let $P_i$ be the parabolic subgroup of $G$ containing $B$ whose Levi subgroup has the root system  $\{\alpha_i,-\alpha_i\}$.  Using the natural Hopf algebra maps  $\She(G)\to \She(B)$ and $\She(P_i) \to \She(B)$ we can get restriction functors 
$$\text{Res}_i:\ \rep(P_i) \to \rep(B), \text{ and } \text{Res}:\ \rep(G) \to \rep(B). 
$$
The restriction functors are exact and naturally commute with taking tensor product of representations. Let $H$ be a subgroup of $G$ and $M \in \rep (G)$. Define the $H$-invariant part of $M$ to be $M^H:=\Hom_{\rep(H)}(k,M)$. Consider the induction functors
\begin{align} 
  &\text{Ind}_i: \rep(B) \to \rep(P_i), \qquad M \mapsto (\She(P_i) \otimes M)^B,\\
	&\text{Ind}: \rep(B) \to \rep(G), \qquad M \mapsto (\She(G) \otimes M)^B.
\end{align}
Set $\Delta_i :=\text{Res}_i \circ \text{Ind}_i \act \rep(B)$ and $\Delta :=\text{Res} \circ \text{Ind} \act \rep(B)$. Notice that $\Delta_i$ and $\Delta$ are left exact, since the induction functors are left exact.

\subsection{The derived categories.} For an algebraic group $H$, the regular comodule $\She(H)$ is injective in $\rep(H)$, moreover for any  $M\in\rep(H)$ the coaction map $M\to \She(H)\otimes M$ provides an embedding of $M$ into an injective object. In particular, $\rep(H)$ has enough injectives. The algebraic De Rham complex $\Omega^\bullet(H)$ provides an injective resolution for the trivial comodule, of the length equal to the dimension of $H$. For any  $M\in\rep(H)$ the complex $\Omega^\bullet(H)\otimes M$ provides an injective resolution for $M$ of the same length.

 Consider now the bounded derived categories $D^b(\rep(B)), D^b(\rep(P_i))$ and $D^b(\rep(G))$.
Let $L_i$ and $L$ be the derived functors of $\text{Res}_i$ and $\text{Res}$ respectively. Denote the right derived functors of $\text{Ind}_i$ and $\text{Ind}$ by $I_i$ and $I$ respectively. Let $D_i = L_i\circ I_i$ and $D=L \circ I$ be the right derived functors of $\Delta_i$ and $\Delta$ respectively.

\begin{prop} \label{properties}
\vskip 1mm
\noindent
\begin{enumerate}[(a)]
\item The functors $L_i$ and $L$ are left adjoint to $I_i$ and $I$ respectively. \label{adjoint}
\item For $M \in D^b(\rep(B))$ and $N \in D^b(\rep(P_i))$ (resp., for $M \in D^b(\rep(B))$ and $N \in D^b(\rep(G))$) we have the tensor identities: \label{tensor}
\[ I_i(M \otimes L_i(N)) \iso I_i(M) \otimes N. \quad (\text{ resp., } I(M \otimes L(N)) \iso I(M) \otimes N).\]
\item The functors $I_i$ and $I$ take the trivial $\She(B)$-comodule to the trivial $\She(P_i)$-comodule (resp., to the trivial $\She(G)$-comodule). \label{trivial}
\item $D_i$ and $D$ are comonads for which the comonad maps $D_i \to D_i^2$ and $D \to D^2$ are isomorphisms.
\end{enumerate}
\end{prop} 
\begin{proof}
The statements corresponding to (\ref{adjoint}) and (\ref{tensor}) for Res and Ind (resp. $\text{Res}_i$ and $\text{Ind}_i$) are proposition 3.4 and 3.6 in \cite{Jant}. The derived functors of a pair of adjoint functors are adjoint. (\ref{tensor}) also follows from thes statement for the non-derived functors since tensoring over a field is exact.
\[ I_i( \text{id} \otimes L_i) \iso R(\text{Ind}_i(\text{id} \otimes \text{Res}_i) \iso R(\text{Ind}_i \otimes \text{id}) \iso I_i \otimes \text{id}. \]
By (\ref{adjoint}) $D_i=L_i \circ I_i$ and $D=L \circ I$ are comonads (See \cite[section VI.1]{Mac}). (\ref{tensor}) and (\ref{trivial}) implies that $I_i \circ L_i(N) \iso N$ for $N \in D^b(\rep(P_i))$ and $I\circ L(N) \iso N$ for $N \in D^b(\rep(G))$. Thus, $\text{Id} \isomap I_i \circ L_i$ (resp. $\text{Id} \isomap I \circ L$) and from this we get the desired isomorphism
\[ D_i=L_i \circ I_i = L_i \circ \text{Id} \circ I_i \isomap L_i \circ I_i \circ L_i \circ I_i=D_i^2. \]
and likewise for $D$.
\end{proof}

\begin{rem} 
It follows that the restriction functors $L_i$ and $L$ are fully faithful.
\end{rem}

\section{Demazure descent.}
Fix a root data $(I,X,Y)$ of the finite type, with the Weyl group $W$ and the braid monoid $\text{Br}^+$. Consider a triangulated category $\ccal$. 

\begin{defin}
A weak braid monoid action on the category $\ccal$ is a collection of triangulated functors
\[ D_w : \ccal \to \ccal, \qquad w \in W \]
satisfying  braid monoid relations, i.e. for all $w_1, w_2 \in W$ there exist isomorphisms of functors  \[ D_{w_1} \circ D_{w_2} \iso D_{w_1w_2}, \qquad \text{if } \ell(w_1 w_2)= \ell(w_1)+\ell(w_2). \]
\end{defin}
Notice that we neither fix the braid relations isomorphisms nor impose any additional relations on them. 

\begin{defin}
Demazure descent data on the category $\ccal$ is a weak braid monoid action $\{D_w\}$ such that for each simple root $s_i$ the corresponding functor $D_{s_i}$ is a comonad for which the comonad map $D_{s_i} \to D_{s_i}^2$ is an isomorphism.
\end{defin}

Here is the central construction of the paper. Consider a triangulated category $\ccal$ with a fixed Demazure descent data $\{D_w,w\in W\}$ of the type $(I,X,Y)$.

\begin{defin}
The descent category  $\desc(\ccal,D_w,w\in W)$ is the full subcategory in $\ccal$ consisting of objects $M$ such that for all $i$ the cones of the counit maps 
$D_{s_i}(M) \stackrel{\epsilon}{\to} M$ are isomorphic to $0$. 
\end{defin}

\begin{rem}
Suppose that $\ccal$ has functorial cones. Then  $\desc(\ccal,D_w,w\in W)$ a full triangulated subcategory in $\ccal$ being the intersection of kernels of $\text{Cone}(D_{s_i} \to Id)$. However, one can prove this statement not using functoriality of cones. 
\end{rem}

\begin{lemma}
An object $M \in \desc(\ccal,D_w,w\in W)$ is naturally a comodule over each $D_{s_i}$.
\end{lemma}
\begin{proof}
By definition the comonad maps
\[ \eta: D_{s_i} \to D_{s_i}^2, \qquad \epsilon: D_{s_i} \to \text{Id} \]
makes the following diagram commutative
\begin{align}
	\xymatrix{& D_{s_i} \ar@{=}[ld] \ar[d]_\eta \\ \text{Id} \circ D_{s_i} & D_{s_i}^2 \ar[l]^(0.4){\epsilon \circ D_{s_i}}}
\end{align}
For Demazure descent data we require that $\eta$ is an isomorphism, so $\epsilon \circ D_{s_i}$ is also an isomorphism. Let $M \in \desc(\ccal,D_w,w\in W)$. That $\text{Cone}(D_{s_i}(M) \stackrel{\epsilon}{\to} M)$ is isomorphic to 0 is equivalent to saying that $D_{s_i}(M) \stackrel{\epsilon}{\to} M$ is an isomorphism. This gives the commutative diagram.
\begin{align}
	\xymatrix{M \ar[r]^{\epsilon^{-1}} \ar[d]_{\epsilon^{-1}} & D_{s_i}(M) \ar@{=}[ld] \ar[d]_\eta \\ D_{s_i}(M) \ar[r]_{(\epsilon \circ D_{s_i})^{-1}} & D_{s_i}^2(M)}
\end{align}
Thus, $\epsilon^{-1}$ satisfies the axiom for the coaction.
\end{proof}
\begin{rem}
Recall that in the usual descent setting either in Algebraic Geometry or in abstract Category Theory (Barr-Beck theorem) descent data includes a pair of adjoint functors and their composition which is a comonad. By definition, the descent category for such data is the category of comodules over this comonad. Our definition of $\desc(C,D_w,w\in W)$  for Demazure  descent data formally is not about comodules, yet the previous Lemma demonstrates that every object of $\desc(C,D_w,w\in W)$ is naturally  equipped with structures of a comodule over each $D_i$ and any morphism in $\desc(C,D_w,w\in W)$ is a morphism of $D_i$-comodules.
\end{rem}

\section{Main Theorem}

We now go back to considering $D_i=L_i \circ I_i$ and $D=L \circ I$.

\begin{prop}
Let $w \in W$ and let $w=s_{i_1} \cdots s_{i_n}$ be a reduced expression. Then $D_w:= D_{i_1} \circ \cdots \circ D_{i_n}$ is independent of the choice of reduced expression and the $D_w$'s form Demazure descent data on $\ccal=D^b(\rep(B))$.
\end{prop}

\begin{lemma} \label{prod Pi}
Let $w=s_{i_1} \cdots s_{i_n}$ be a reduced expression. Then
\[ P_{i_1} \cdots P_{i_n}=\bigcup_{w' \leq w} Bw'B, \]
where the union is over all $w' \in W$ which is $\leq w$ in the Bruhat order.
\end{lemma}

\begin{proof}
The proof goes by induction on $n=\ell(w)$. It is true for $n=1$ by definition of $P_i$. Set $v=s_{i_1} \cdots s_{i_{n-1}}$. Using the hypotheses we get
\begin{align}
	P_{i_1} \cdots P_{i_{n-1}} P_{i_n}=\Bigl( \bigcup_{w' \leq v} Bw'B \Bigr) (B \cup Bs_{i_n}B)=\bigcup_{w' \leq v} Bw'B \cup \bigcup_{w' \leq v} (Bw'B)(Bs_{i_n}B)
\end{align}
Let $w'$ be any element in $W$ and $s$ a simple reflection. Then by \cite[Cor. 28.3]{Hum} we have $(Bw'B)(BsB) \subseteq Bw'sB \cup Bw'B$. Thus, if $w's_{i_n} \leq w' \leq v$ then $(Bw'B)(Bs_{i_n}B)$ is contained in the first union. If $w' \leq w's_{i_n}$ then we have $(Bw'B)(Bs_{i_n}B)=Bw's_{i_n}B$ by \cite[Lemma 29.3A and section 29.1]{Hum}. Thus, the product can be written as
\begin{align}
	P_{i_1} \cdots P_{i_n}&=\bigcup_{w' \leq v} Bw'B \enskip \cup \smashoperator[r]{\bigcup_{\substack{w' \leq v\\ w' \leq w's_{i_n}}}} \; Bw's_{i_n}B\\
	&=\bigcup_{w' \leq v} Bw'B \enskip \cup \smashoperator[r]{\bigcup_{\substack{w''s_{i_n} \leq v, \\ w''s_{i_n} \leq w''}}} \; Bw''B
\end{align}
\begin{claim}
The conditions $w''s_{i_n} \leq v$ and $w''s_{i_n} \leq w''$ is equivalent to the conditions $w'' \leq w$ and $w''s_{i_n} \leq w''$.
\end{claim}
\begin{proof}[Proof of the claim]
Assume that $w''s_{i_n} \leq v$. By \cite[Prop. 5.9]{Hum2} this implies that $w'' \leq v$ or $w'' \leq vs_{i_n}=w$. In both cases we get $w'' \leq w$ since $v \leq w$. Assume now that $w'' \leq w$ and $w''s_{i_n} \leq w''$. $w''$ has a reduced expression of the form
\[ w''=s_{i_1} \cdots \hat{s}_{i_{j_1}} \cdots \hat{s}_{i_{j_2}} \cdots \hat{s}_{i_{j_k}} \cdots s_{i_n}, \]
where the $\hat{}$ indicates that the term has been removed from the product. If $j_k \neq n$ then
\[ w''s_{i_n}=s_{i_1} \cdots \hat{s}_{i_{j_1}} \cdots \hat{s}_{i_{j_2}} \cdots \hat{s}_{i_{j_k}} \cdots s_{i_{n-1}} \leq s_{i_1} \cdots s_{i_{n-1}}=v. \]
If $j_k =n$ then $w'' \leq v$. Since $w''s_{i_n} \leq w''$ by assumption we get $w''s_{i_n} \leq v$.
\end{proof}

If $w' \leq v$ in the first union satisfies that $w's_{i_n} \leq w'$ then it is also contained in the second union. Using the claim we get
\[P_{i_1} \cdots P_{i_n}=\smashoperator[r]{\bigcup_{\substack{w' \leq v \\ w' \leq w's_{i_n}}}} \; Bw'B \enskip \cup \smashoperator[r]{\bigcup_{\substack{w'' \leq w, \\ w''s_{i_n} \leq w''}}} \; Bw''B\]
Assume that $w' \leq w$ and $w' \leq w's_{i_n}$. Then $w'$ has a reduced expression of the form
\[ w'=s_{i_1} \cdots \hat{s}_{i_{j_1}} \cdots \hat{s}_{i_{j_2}} \cdots \hat{s}_{i_{j_k}} \cdots s_{i_n}. \]
If $j_k=n$ then $w' \leq v$. If $j_k \neq n$ then $w's_{i_n} \leq v$, but since $w' \leq w's_{i_n}$ we get $w' \leq v$. Hence, the conditions $w' \leq v$ and $w' \leq w's_{i_n}$ can be replaced by $w' \leq w$ and $w' \leq w's_{i_n}$. Thus,
\begin{align}
	P_{i_1} \cdots P_{i_n}=\smashoperator[r]{\bigcup_{\substack{w' \leq w\\ w' \leq w's_{i_n}}}} \; Bw'B \enskip \cup \smashoperator[r]{\bigcup_{\substack{w'' \leq w, \\ w''s_{i_n} \leq w''}}} \; Bw''B=\bigcup_{w' \leq w} Bw'B.
\end{align}
This finishes the induction step.
\end{proof}

\begin{proof}[Proof of the proposition]
Let $w \in W$ and let $s_{i_1} \cdots s_{i_n}=s_{j_i} \cdots s_{j_n}$ be two reduced expressions for $w$. By lemma \ref{prod Pi} this implies that $P_{i_1} \cdots P_{i_n}=P_{j_1} \cdots P_{j_n}$. By \cite[Thm. 3.1]{CPS} the $B$-module structure of $\Delta_{i_1} \circ \cdots \circ \Delta_{i_n}$ is determined up to a natural isomorphism by the set $P_{i_1} \cdots P_{i_n}$. Hence
\[ \Delta_{i_1} \circ \cdots \circ \Delta_{i_n} \iso \Delta_{j_1} \circ \cdots \circ \Delta_{j_n} \]
Hence, for any choice of reduced expression we can define
\[ \Delta_w :=\Delta_{i_1} \circ \cdots \circ \Delta_{i_m}. \]
Let $w_1$ and $w_2$ be elements in $W$ such that $\ell(w_1 w_2)=\ell(w_1)+\ell(w_2)$. Pick reduced expressions $s_{i_1} \cdots s_{i_r}$ and $s_{j_1} \cdots s_{j_t}$ for $w_1$ and $w_2$ respectively. Then $s_{i_1} \cdots s_{i_r} s_{j_1} \cdots s_{j_t}$ is a reduced expression for $w_1w_2$ and we get braid relations for the $\Delta_w$
\[ \Delta_{w_1} \circ \Delta_{w_2}=\Delta_{i_1} \circ \cdots \circ \Delta_{i_r} \circ \Delta_{j_1} \circ \cdots \circ \Delta_{j_t}=\Delta_{w_1 w_2}. \]
Define
\[ D_w:=R(\Delta_w)=R(\Delta_{i_1} \circ \cdots \circ \Delta_{i_m})=R(\Delta_{i_1}) \circ \cdots \circ R(\Delta_{i_m})=D_{i_1} \circ \cdots \circ D_{i_m} \]
The braid relations for $D_w$ now follows from the braid relations for $\Delta_w$
\[ D_{w_1} \circ D_{w_2} = R(\Delta_{w_1}) \circ R(\Delta_{w_2}) \iso R(\Delta_{w_1} \circ \Delta_{w_2}) \iso R(\Delta_{w_1 w_2}) = D_{w_1w_2}. \qedhere\]
\end{proof}

\begin{thm}
$\desc(\ccal,D_w,w\in W)$ is equivalent to $D^b(\rep(G))$
\end{thm}
\begin{proof}
Let $M \in D^b(\rep(B))$. Being able to extend $M$ to an element in $D^b(\rep(G)$ is equivalent to $M$ being in the image of $L$. Assume $M=L(N)$ for some $N \in D^b(\rep(G))$. Then $D(M)=L \circ I \circ L(N) \iso L(N)=M$. If $D(M) \isomap M$ then $M \iso L(I(M))$, so $M$ is in the image of $L$. Thus, being in the image of $L$ is equivalent to $D(M) \to M$ being an isomorphism which is again equivalent to  $M \in \text{ker}(C)$, where $C:=\text{Cone}(D \to \text{Id})$. Set $C_i:=\text{Cone}(D_i \to \text{Id})$.

\begin{claim}
$\ker(C) = \bigcap_i \ker(C_i)$
\end{claim}
\begin{proof}[Proof of claim]
Assume that $M \in \ker(C)$. Then $M=L(N)$ for some $N \in D^b(\rep(G))$. But then $M=L_i(N_{|P_i})$ for all $i$, so $D_i(M) \to M$ is an isomorphism for all $i$. Hence, $M \in \bigcap_i \ker(C_i)$. Assume that $M \in \cap_i \ker(C_i)$. Then all $D_i(M) \to M$ are isomorphisms. Choose a reduced expression $s_{i_1} \cdots s_{i_N}$ for the longest element in the Weyl group. Then $P_{i_1} \cdots P_{i_N}=G$. By \cite{CPS} we have $D=D_{i_1} \circ \cdots \circ D_{i_N}$.
\begin{align}
D(M) &\iso D_{i_1} \circ \cdots \circ D_{i_N}(M)\\
& \iso D_{i_1} \circ \cdots \circ D_{i_{N-1}}(M) \iso \dots\\
& \iso D_{i_1}(M) \iso M.
\end{align}
Hence,
\[ \text{Cone}(D(M) \to M) \iso \text{Cone}\bigl(D(D(M)) \to D(M)\bigr). \]
By definition of a comonad we have the following commutative diagram
\begin{align}
	\xymatrix{& D \ar@{=}[ld] \ar[d]_\eta \\ \text{Id} \circ D & D^2 \ar[l]^(0.4){\epsilon D}}
\end{align}
Since $\eta$ is an isomorphism so is $\epsilon D$ and thus $\text{Cone}\bigl(D(D(M)) \to D(M)\bigr)=0$. This shows that $M \in \ker(C)$.
\end{proof}
From the claim we get that
\[ D^b(\rep(G))=\bigcap_i \ker(C_i), \]
which is exactly the descent category.
\end{proof}

\section{Further directions}

\subsection{Quantum groups.}   Fix a root data $(I,X,Y)$ of the finite type. Let $\mathsf{U}_{\mathcal{A}}$ be the Lusztig quantum group over the ring of quantum integers $\mathcal{A} =\mathbb{Z}[v,v^{-1}]$. Denote the quantum Borel subalgebra by $\mathsf{B}_{\mathcal{A}}$. For a simple root $\alpha_i$ the corresponding quantum parabolic sub algebra is denoted by $\mathsf{P}_{i,\mathcal{A}}$.

Following \cite{APK} we consider the categories of locally finite weight modules over $\mathsf{U}_{\mathcal{A}}$ (resp. over $\mathsf{B}_{\mathcal{A}}$, resp. over $\mathsf{P}_{i,\mathcal{A}}$) denoted by $\rep(\mathsf{U}_{\mathcal{A}})$ (resp. by $\rep(\mathsf{B}_{\mathcal{A}})$, resp. by $\rep(\mathsf{P}_{i,\mathcal{A}})$).
 We consider the corresponding  derived categories $D^b(
\rep(\mathsf{U}_{\mathcal{A}}))$, $D^b(
\rep(\mathsf{B}_{\mathcal{A}}))$ and $D^b(
\rep(\mathsf{P}_{i,\mathcal{A}}))$.

 Like in the reductive algebraic group case, the restriction functors 
 $$
 L:\ D^b(
\rep(\mathsf{U}_{\mathcal{A}}))\to D^b(
\rep(\mathsf{B}_{\mathcal{A}})) \text{ and } L_i:\ D^b(
\rep(\mathsf{P}_{i, \mathcal{A}}))\to D^b(
\rep(\mathsf{B}_{\mathcal{A}})) 
$$
are fully faithful and possess right adjoint functors denoted by $I$ (resp. by $I_i$).  Denote the comonad  $L_i \circ I_i$ by $D_i$. Andersen, Polo and Wen proved that the functors $D_i$ define a weak braid monoid action on the category $D^b(
\rep(\mathsf{B}_{\mathcal{A}}))$. One can easily prove that the functors form Demazure descent data. The corresponding descent category $\mathfrak{Desc}(D^b(
\rep(\mathsf{B}_{\mathcal{A}})), D_1,\ldots D_n)$ is equivalent to $D^b(
\rep(\mathsf{U}_{\mathcal{A}}))$.

\subsection{Equivariant sheaves.} Let $X$ be an affine scheme equipped with an action of a reductive algebraic group $G$. Fix a Borel subgroup $B\subset G$. Like in the main body of the present paper, consider the minimal parabolic subgroups in $G$ denoted by  $P_1,\ldots P_n$.
Denote the derived categories of quasicoherent sheaves on $X$ equivariant with respect to $G$ (resp., $B$, resp., $P_i$) by $D^b(QCoh^G(X))$ (resp., by $D^b(QCoh^B(X))$, resp. by $D^b(QCoh^{P_i}(X))$). We have the natural functors provided by restriction of equivariance
$
L:\ D^b(QCoh^G(X))\to D^b(QCoh^B(X)) \text{ and } L_i:\ D^b(QCoh^{P_i}(X))\to D^b(QCoh^B(X)).
$
These functors have the right adjoint ones $I$, resp. $I_1,\ldots I_n$. The comonads $D_1,\ldots D_n$ given by the  compositions of extension and restriction of equivariance define a Demazure descent data on the category $D^b(QCoh^B(X))$. The corresponding descent category is equivalent to $D^b(QCoh^G(X))$.

\subsection{Algebraic loop group}
For a simple algebraic group $G$ consider the algebraic loop group $LG=\text{Map}(\overset{\bullet}{D}, G)$ (resp. the formal arcs group $L^+G=\text{Map}(D,G)$). Here $D$ (resp. $\overset{\bullet}{D}$) denotes the formal disc (resp. the formal punctured disc). Consider the affine Kac-Moody central extension
\[ 1 \to \mathbb{G}_m \to \widehat{LG} \to LG \to 1. \]
The affine analog of the Borel subgroup $B \subset G$ is the Iwahori subgroup $Iw \subset L^+G$. Let $P_0, \dots, P_n$ be the standard minimal parahoric subgroups in $L^+G$. One considers the adjoint pairs of coinduction-restriction functors $I_0,L_0, \dots, I_n, L_n$ between  $D^b(\rep(Iw))$ and $D^b(\rep(P_i))$. Denote the comonads $L_i \circ I_i$ by $D_i$ for $i=0, \dots, n$. We claim that $D_0, \dots, D_n$ form affine Demazure descent data on $D^b(\rep(Iw))$. We conjecture that the descent category is equivalent to $D^b(\rep(\widehat{LG}))$ (direct sum of the categories over all positive integral levels).

\end{document}